\documentclass[12pt,reqno]{amsart}

\usepackage{amsmath,amsfonts,amssymb}

\numberwithin{equation}{section}

\newtheorem{theorem}{Theorem}[section]
\newtheorem{lemma}[theorem]{Lemma}
\newtheorem{proposition}[theorem]{Proposition}
\newtheorem{corollary}[theorem]{Corollary}

\theoremstyle{definition}
\newtheorem{definition}[theorem]{Definition}

\begin{document}

\baselineskip=15.5pt

\title[Homogeneous bundles and trivial logarithmic tangent bundle]{Homogeneous principal bundles over
manifolds with trivial logarithmic tangent bundle}

\author[H. Azad]{Hassan Azad}

\address{Abdus Salam School of Mathematical Sciences, GC University Lahore, 68-B, New Muslim
Town, Lahore 54600, Pakistan}

\email{hassan.azad@sms.edu.pk}

\author[I. Biswas]{Indranil Biswas}

\address{School of Mathematics, Tata Institute of Fundamental
Research, Homi Bhabha Road, Mumbai 400005, India}

\email{indranil@math.tifr.res.in}

\author[M. A. Khadam]{M. Azeem Khadam}

\address{Abdus Salam School of Mathematical Sciences, GC University Lahore, 68-B, New Muslim
Town, Lahore 54600, Pakistan}

\email{azeem.khadam@sms.edu.pk}



\thanks{The second-named author is partially supported by a J. C. Bose Fellowship. The third-named
author is grateful to ASSMS, GC University Lahore for the support of this research under
the postdoctoral fellowship.}

\begin{abstract}
Winkelmann considered compact complex manifolds $X$ equipped with a reduced effective normal crossing
divisor $D\, \subset\, X$ such that the logarithmic tangent bundle $TX(-\log D)$ is
holomorphically trivial. He characterized them as pairs $(X,\, D)$ admitting a holomorphic action of a
complex Lie group $\mathbb G$ satisfying certain conditions \cite{Wi1}, \cite{Wi2}; this $\mathbb G$ is
the connected component, containing the identity element, of the group of holomorphic automorphisms of $X$
that preserve $D$. We characterize the homogeneous holomorphic principal $H$--bundles over $X$, where $H$
is a connected complex Lie group. Our characterization says that the following three are equivalent:\\

(1)~ $E_H$ is homogeneous.\\

(2)~ $E_H$ admits a logarithmic connection singular over $D$.\\

(3)~ The family of principal $H$--bundles $\{g^*E_H\}_{g\in \mathbb G}$ is infinitesimally
rigid at the identity element of the group $\mathbb G$.\\
\end{abstract}

\maketitle

\section{Introduction}\label{se1}

The present work was motivated by a work of Winkelmann \cite{Wi2}. We begin by very briefly recalling from \cite{Wi2}.
Let $X$ be a compact connected complex manifold and $D\, \subset\, X$ a reduced effective normal crossing
divisor (definition will be recalled in Section \ref{sec2.1}), such that the corresponding
logarithmic tangent bundle $TX(-\log D)$ is holomorphically trivial.
Let $\mathbb G$ be the connected component of the 
group of holomorphic automorphisms of $X$ that preserve $D$. This $\mathbb G$ is a complex Lie group 
that acts transitively on the complement $X\setminus D$. The tautological action of $\mathbb G$ on $X$
has certain properties (they are recalled in Section \ref{se3.1}); these properties actually characterize pairs
$(X,\, D)$ of the above type \cite[p.~196, Theorem 1]{Wi2}.

Let $H$ be a connected complex Lie group and $E_H$ a holomorphic principal $H$--bundle over $X$. It is called
homogeneous if for every $g\, \in\, {\mathbb G}$ the pulled back principal $H$--bundle $g^*E_H$ is holomorphically 
isomorphic to $E_H$ (Definition \ref{def1} and Proposition \ref{prop2}).

Let $\varpi\, :\, {\mathbb G}\times X\, \longrightarrow\, X$ be the evaluation map
defined by $(g,\, x)\, \longmapsto\, g(x)$. Let
$$
{\mathcal E}_H\, :=\, \varpi^*E_H\, \longrightarrow\, {\mathbb G}\times X
$$
be the holomorphic principal $H$--bundle over ${\mathbb G}\times X$ obtained by pulling back $E_H$
using the above holomorphic map $\varpi$. Consider ${\mathcal E}_H$ as a holomorphic family of holomorphic
principal $H$--bundles over $X$ parametrized by $\mathbb G$. Let
$$
f\, :\, \text{Lie}({\mathbb G})\,=\, T_e{\mathbb G}\, \longrightarrow\, H^1(X,\, \text{ad}(E_H))
$$
be the infinitesimal deformation map at the identity element $e\, \in\, \mathbb G$ for this family
of holomorphic principal $H$--bundles over $X$.

Let $E_H$ be any holomorphic principal $H$--bundle over $X$.
We prove that the following three statements are equivalent:
\begin{enumerate}
\item {\it $E_H$ is homogeneous}.

\item {\it $E_H$ admits a logarithmic connection singular over $D$}.

\item {\it The above homomorphism $f$ vanishes identically (meaning the above family of principal $H$--bundles
is infinitesimally rigid at $e\, \in\, \mathbb G$)}.
\end{enumerate}
Corollary \ref{cor2} says that the first two statements are equivalent.
Proposition \ref{prop1} says that the first and the third statements are equivalent.

\section{Logarithmic connections on a holomorphic principal bundle}

\subsection{Logarithmic differential forms}\label{sec2.1}

Let $X$ be a complex manifold of complex dimension $d$. A reduced effective divisor
$D\, \subset\, X$ is said to be a \textit{normal crossing divisor} if for every point $x\, \in\, D$
there are holomorphic coordinate functions $z_1,\, \cdots ,\, z_d$ defined on an open
neighborhood $U\, \subset\, X$ of $x$ with $z_1(x)\,=\, \cdots\,=\, z_d(x)\,=\, 0$,
and there is an integer $1\, \leq\, k\, \leq\, d$, such that
\begin{equation}\label{e1}
D\cap U\, =\, \{y\, \in\, U \, \mid\, z_1(y)\,=\, \cdots \,=\, z_k(y)\,=\, 0\}\, .
\end{equation}
Note that it is not assumed that the irreducible components of $D$ are smooth. In \cite{Wi1}
and \cite{Wi2}, the terminology ``locally simple normal crossing divisor'' is used; however, it seems
that ``normal crossing divisor'' is used more often in the literature; see \cite{Co}.

The holomorphic cotangent and tangent bundles of $X$ will be denoted by $\Omega^1_X$ and $TX$
respectively. Take a normal crossing divisor $D$ on $X$. Let
\begin{equation}\label{e2}
TX(-\log D)\, \subset\, TX
\end{equation}
be the coherent analytic subsheaf generated by all locally defined holomorphic vector fields $v$ on $X$
such that $v({\mathcal O}_X(-D))\, \subset\, {\mathcal O}_X(-D)$. In other words, if $v$
is a holomorphic vector field defined over $U\, \subset\, X$, then
$v$ is a section of $TX(-\log D)\vert_U$ if and only if $v(f)\vert_{U\cap D}\,=\, 0$ for
all holomorphic functions $f$ on $U$ that vanish on $U\cap D$. It is straightforward to check that
the stalk of sections of $TX(-\log D)$ at the point $x$ in \eqref{e1} is generated by
$$
z_1\frac{\partial}{\partial z_1},\, \cdots,\, z_k\frac{\partial}{\partial z_k},\,
\frac{\partial}{\partial z_{k+1}} ,\, \cdots,\, \frac{\partial}{\partial z_d}\, .
$$
The condition that $D$ is a normal crossing divisor implies that the coherent analytic sheaf
$TX(-\log D)$ is in fact locally free. Clearly, we have $TX(-\log D)\vert_{X\setminus D}\,=\,
TX\vert_{X\setminus D}$. This vector bundle $TX(-\log D)$ is called the \textit{logarithmic
tangent bundle} for the pair $(X,\, D)$.

Restricting the natural homomorphism $TX(-\log D)\, \longrightarrow\, TX$ to the divisor $D$,
we get a homomorphism
$$
\psi\, :\, TX(-\log D)\vert_D\, \longrightarrow\, TX\vert_D\, .
$$
Let
\begin{equation}\label{f1}
{\mathbb L}\, :=\, \text{kernel}(\psi)\, \subset\, TX(-\log D)\vert_D
\end{equation}
be the kernel. To describe $\mathbb L$, let
$$
\nu\, :\, \widetilde{D}\, \longrightarrow\, D
$$
be the normalization; the given condition on $D$ implies that $\widetilde{D}$ is smooth.
Then $\mathbb L$ is identified with the direct image
\begin{equation}\label{f0}
{\mathbb L}\, =\, \nu_*{\mathcal O}_{\widetilde{D}}\, .
\end{equation}
The key point in the construction of this isomorphism is the following: Let $Y$ be a Riemann surface and
$y_0\, \in\, Y$ a point; then for any holomorphic coordinate function $z$ around $y_0$, with $z(y_0)\,=\,0$,
the evaluation of the local section $z\frac{\partial}{\partial z}$ of $TY\otimes {\mathcal O}_Y(-y_0)$
at the point $y_0$ does not depend on the choice of the coordinate function $z$.

Consider the Lie bracket operation on the locally defined holomorphic vector fields
on $X$. The holomorphic sections of $TX(-\log D)$ are closed under this Lie bracket. Indeed, if
$v_1,\, v_2$ are holomorphic sections of $TX(-\log D)$ over $U\, \subset\, X$, and $f$
is a holomorphic function on $U$ that vanishes on $U\cap D$, then from the identity
$$
[v_1,\, v_2](f)\, =\, v_1(v_2(f)) - v_2(v_1(f))
$$
we conclude that the function $[v_1,\, v_2](f)$ vanishes on $U\cap D$.

The dual vector bundle $TX(-\log D)^*$ is denoted by
$\Omega^1_X(\log D)$. From \eqref{e2} it follows that
$$
(TX)^* \,=\, \Omega^1_X\, \subset\, \Omega^1_X(\log D)\, .
$$
The stalk of sections of $\Omega^1_X(\log D)$ at the point $x$ in \eqref{e1} is generated by
$$
\frac{1}{z_1}dz_1,\, \cdots,\, \frac{1}{z_k}dz_k,\,
dz_{k+1} ,\, \cdots,\, dz_d\, .
$$
For every integer $i \, \geq\, 0$, define $\Omega^i_X(\log D)\,:=\, \bigwedge^i \Omega^1_X(\log D)$.

Let
$$
\eta\, :\, D\, \longrightarrow\, X
$$
be the inclusion map. Taking dual of the homomorphism $\psi$ (see \eqref{f1}), and using
\eqref{f0}, we get the following short exact sequence of coherent analytic sheaves on $X$
$$
0\, \longrightarrow\, \Omega^1_X\,\longrightarrow\, \Omega^1_X(\log D)\,
\stackrel{\mathcal R}{\longrightarrow}\, (\eta\circ\nu)_*{\mathcal O}_{\widetilde{D}}\, 
\longrightarrow\, 0
\, ,
$$
where $\nu$ is the map in \eqref{f0}; the above homomorphism $\mathcal R$
is known as the \textit{residue map}.

We refer the reader to \cite{Sa} for more details on logarithmic forms and vector
fields.

\subsection{Atiyah bundle and logarithmic connection}

Let $H$ be a complex Lie group. The Lie algebra of $H$ will be denoted by $\mathfrak h$.
Let
\begin{equation}\label{e3}
p\, :\, E_H\, \longrightarrow\, X
\end{equation}
be a holomorphic principal $H$--bundle; we recall that this means that $E_H$ is a
holomorphic fiber bundle over $X$ equipped with a holomorphic right-action of the group $H$
\begin{equation}\label{e4}
q'\, :\, E_H\times H\,\longrightarrow\, E_H
\end{equation}
such that $p(q'(z,\, h))\,=\, p(z)$ for all $(z,\, h)\, \in\, E_H\times H$, where
$p$ is the projection in \eqref{e3} and, furthermore, the resulting
map to the fiber product
$$
E_H\times H \, \longrightarrow\, E_H\times_X E_H\, , \ \ (z,\, h) \, \longrightarrow\, (z,\,
q'(z,\, h))
$$
is a biholomorphism. For notational convenience, the point
$q'(z,\, h)\, \in\, E_H$, where $(z,\, h)\,\in\, E_H\times H$, will be denoted by $zh$.

As before, let $D\, \subset\, X$ be a normal crossing divisor. Since $p$ in \eqref{e3} is
a holomorphic submersion, the inverse image
\begin{equation}\label{e5}
\widehat{D} \,:=\, p^{-1}(D)\, \subset\, E_H
\end{equation}
is also a normal crossing divisor. Consider the action of $H$ on the tangent bundle $TE_H$
given by the action of $H$ on $E_H$ in \eqref{e4}. This action of $H$ on $TE_H$ clearly
preserves the subsheaf $TE_H(-\log \widehat{D})\, \subset\, TE_H$. The corresponding quotient
\begin{equation}\label{e6}
{\rm At}(E_H)(-\log D)\, :=\, TE_H(-\log \widehat{D})/H \, \longrightarrow\,
E_H/H \,=\, X
\end{equation}
is evidently a holomorphic vector bundle over $X$; it is called the \textit{logarithmic
Atiyah bundle} (see \cite{At} for the case where $D$ is the zero divisor).

Let $dp\, :\, TE_H\, \longrightarrow\, p^*TX$ be the differential of the projection $p$
in \eqref{e3}. Let
$$
{\mathcal K} \, :=\, \text{kernel}(dp)\, \subset\, TE_H
$$
be the kernel of $dp$. So we have the following short exact sequence of holomorphic vector bundles
on $E_H$:
\begin{equation}\label{f2}
0\, \longrightarrow\, {\mathcal K}\, \longrightarrow\, TE_H\, \stackrel{dp}{\longrightarrow}\,
p^*TX \, \longrightarrow\,0\, .
\end{equation}
Note that we have $${\mathcal K} \,\subset\, TE_H(-\log \widehat{D})\, ,$$ and
$dp(TE_H(-\log \widehat{D}))\,=\, p^*(TX(-\log D))$. Therefore, the
short exact sequence in \eqref{f2} gives the following
short exact sequence of holomorphic vector bundles over $E_H$
\begin{equation}\label{e7}
0\, \longrightarrow\, {\mathcal K}\, \longrightarrow\, TE_H(-\log \widehat{D})\,
\stackrel{dp}{\longrightarrow}\, p^*(TX(-\log D))\, \longrightarrow\,0
\end{equation}
(the restriction of $dp$ to $TE_H(-\log \widehat{D})$ is also denoted by $dp$).
The above action of $H$ on $TE_H$ clearly preserves the subbundle $\mathcal K$. The quotient
$$
\text{ad}(E_H)\,:=\, {\mathcal K}/H \,\, \longrightarrow\, E_H/H\,=\, X
$$
is called the \textit{adjoint vector bundle} for $E_H$. We note that $\text{ad}(E_H)$ is identified
with the holomorphic vector bundle $E_H\times^H \mathfrak h\, \longrightarrow\, X$ associated to the
principal $H$--bundle $E_H$ for the adjoint action of $H$ on
the Lie algebra $\mathfrak h$. This isomorphism
between ${\mathcal K}/H$ and $E_H\times^H \mathfrak h$ is obtained from the fact that the
action of $H$ on $E_H$ identifies $\mathcal K$ with the trivial holomorphic vector bundle
$E_H\times \mathfrak h$ over $E_H$ with fiber $\mathfrak h$.

Take quotient of the vector bundles in \eqref{e7} by the actions of $H$. From \eqref{e7} we get a
short exact sequence of holomorphic vector bundles over $X$
\begin{equation}\label{e8}
0\, \longrightarrow\, \text{ad}(E_H)\,:=\,{\mathcal K}/H\, \stackrel{\iota_0}{\longrightarrow}\,
(TE_H(-\log \widehat{D}))/H \,=:\, {\rm At}(E_H)(-\log D)
\end{equation}
$$
\stackrel{\beta}{\longrightarrow}\,
(p^*(TX(-\log D)))/H \,=\, TX(-\log D) \, \longrightarrow\,0\, ;
$$
it is called the \textit{logarithmic Atiyah exact sequence} for $E_H$. The homomorphism $dp$ in
\eqref{e7} descents to the homomorphism $\beta$ in \eqref{e8}.

A \textit{logarithmic connection} on $E_H$ singular over $D$ is a holomorphic homomorphism
of vector bundles
$$
\varphi\, :\, TX(-\log D)\, \longrightarrow\,{\rm At}(E_H)(-\log D)
$$
such that
\begin{equation}\label{e10}
\beta\circ\varphi\,=\, \text{Id}_{TX(-\log D)}\, ,
\end{equation}
where $\beta$ is the projection in \eqref{e8}. In other words, giving a logarithmic connection 
on $E_H$ singular over $D$ is equivalent to giving a holomorphic splitting of the short exact 
sequence in \eqref{e8}. See \cite{De} for logarithmic connections (see also \cite{BHH}).

As noted before, the locally defined holomorphic sections of the
logarithmic tangent bundles $TX(-\log D)$ and $TE_H(-\log 
\widehat{D})$ are closed under the Lie bracket operation of vector fields. The locally defined 
holomorphic sections of the subbundle $\mathcal K$ in \eqref{f2} are clearly closed under the 
Lie bracket operation. The homomorphisms in the exact sequence \eqref{f2} are all compatible 
with the Lie bracket operation. Since the Lie bracket operation commutes with diffeomorphisms, 
for any two $H$--invariant holomorphic vector fields $v,\, w$ defined on an $H$--invariant 
open subset of $E_H$, their Lie bracket $[v,\, w]$ is again holomorphic and $H$--invariant. Therefore, 
the sheaves of sections of the three vector bundles in \eqref{e8} are all equipped
with a Lie bracket 
operation. Moreover, all the homomorphisms in \eqref{e8} commute with these operations.

Take a homomorphism
$$
\varphi\, :\, TX(-\log D)\, \longrightarrow\,{\rm At}(E_H)(-\log D)
$$
satisfying the condition in \eqref{e10}. Then for any two holomorphic sections
$v_1,\, v_2$ of $TX(-\log D)$ over $U\, \subset\, X$, consider
$$
{\mathbb K}(v_1,\, v_2)\, :=\, [\varphi(v_1),\, \varphi(v_2)]- \varphi([v_1,\, v_2])\, .
$$
The projection $\beta$ in \eqref{e8} intertwines the Lie bracket operations
on the sheaves of sections of ${\rm At}(E_H)(-\log D)$ and $TX(-\log D)$, and hence we have
$\beta({\mathbb K}(v_1,\, v_2))\,=\, 0$. Consequently, ${\mathbb K}(v_1,\, v_2)$ is a holomorphic section
of $\text{ad}(E_H)$ over $U$. From the identity $[fv,\, w]\,=\, f[v,\, w]- w(f)\cdot v$, where
$f$ is a holomorphic function while $v$ and $w$ are holomorphic vector fields, it follows that
$$
{\mathbb K}(fv_1,\, v_2) \,=\, f {\mathbb K}(v_1,\, v_2)\, .
$$
Also, we have ${\mathbb K}(v_1,\, v_2)\,=\, -{\mathbb K}(v_2,\, v_1)$. Therefore, the mapping
$(v_1,\, v_2)\, \longmapsto\, {\mathbb K}(fv_1,\, v_2)$ defines a holomorphic section
\begin{equation}\label{e13}
{\mathbb K}(\varphi)\, \in\, H^0(X,\, \Omega^2_X(\log D)\otimes {\rm ad}(E_H))\, .
\end{equation}
The section ${\mathbb K}(\varphi)$ in \eqref{e13} is called the \textit{curvature} of the
logarithmic connection $\varphi$.

\subsection{Residue of a logarithmic connection}

The quotient $(TE_H)/H$ is a holomorphic vector
bundle over $E_H/H\,=\, X$. It is the Atiyah bundle for $E_H$; let ${\rm At}(E_H)$ denote this
Atiyah bundle (see \cite{At}). Taking quotient of the vector bundles in \eqref{f2}
by the actions of $H$, from \eqref{f2} we get a
short exact sequence of holomorphic vector bundles over $X$
$$
0\, \longrightarrow\, \text{ad}(E_H)\,:=\,{\mathcal K}/H\, \longrightarrow\,
(TE_H)/H \,=:\, {\rm At}(E_H)
$$
\begin{equation}\label{f3}
\stackrel{\beta'}{\longrightarrow}\,
(p^*TX)/H \,=\, TX \, \longrightarrow\,0\, ,
\end{equation}
which is known as the Atiyah exact sequence for $E_H$ (see \cite{At});
note that $\beta$ in \eqref{e8} is the restriction of $\beta'$ in \eqref{f3}. Restricting
to $D$ the exact sequences in \eqref{e8} and \eqref{f3}, we get the following commutative diagram
\begin{equation}\label{e9}
\begin{matrix}
0& \longrightarrow & \text{ad}(E_H)\vert_D & \stackrel{\widehat{\iota}_0}{\longrightarrow} &
{\rm At}(E_H)(-\log D)\vert_D & \stackrel{\widehat{\beta}}{\longrightarrow} &
TX(-\log D)\vert_D & \longrightarrow\,0 &\\
&& \Vert && ~\Big\downarrow\mu && ~\Big\downarrow\psi\\
0& \longrightarrow & \text{ad}(E_H)\vert_D & \stackrel{\iota_1}{\longrightarrow} &
{\rm At}(E_H)\vert_D & \stackrel{\widehat{\beta}'}{\longrightarrow} &
TX\vert_D & \longrightarrow\,0 &
\end{matrix}
\end{equation}
whose rows are exact; the map $\psi$ in the one in \eqref{f1} and
$\mu$ is the homomorphism given by the natural homomorphism ${\rm At}(E_H)(-\log D)
\,\longrightarrow\, {\rm At}(E_H)$. In \eqref{e9} the following convention is employed: the restriction
to $D$ of a map on $X$ is denoted by the same symbol after adding a hat. From
\eqref{f1} we know that the kernel of $\psi$ is $\mathbb L\,=\,
\nu_*{\mathcal O}_{\widetilde{D}}$ (see \eqref{f0}). Let
$$
\iota_{\mathbb L}\, :\, {\mathbb L}\, \longrightarrow\, TX(-\log D)\vert_D
$$
be the inclusion map.

Let $\varphi\, :\, TX(-\log D)\, \longrightarrow\, {\rm At}(E_H)(-\log D)$ be a
logarithmic connection on $E_H$ singular over $D$. Consider the composition
$$
\widehat{\varphi}\circ\iota_{\mathbb L}\,:\, {\mathbb L}\, \longrightarrow\,
{\rm At}(E_H)(-\log D)\vert_D
$$
(the restriction of $\varphi$ to $D$ is denoted by $\widehat{\varphi}$). From the
commutativity of the diagram in \eqref{e9} it follows that
\begin{equation}\label{e11}
\widehat{\beta}'\circ\mu\circ \widehat{\varphi}\circ\iota_{\mathbb L}\,=\,
\psi\circ \widehat{\beta}\circ\widehat{\varphi}\circ\iota_{\mathbb L}\,.
\end{equation}
But $\widehat{\beta}\circ\widehat{\varphi}\,=\, \text{Id}_{TX(-\log D)\vert_D}$
by \eqref{e10}, while $\psi\circ\iota_{\mathbb L}\,=\, 0$ by \eqref{f1}, so these together
imply that $$\psi\circ \widehat{\beta}\circ\widehat{\varphi}\circ\iota_{\mathbb L}\,=\, 0\, .$$
Hence from \eqref{e11} we conclude that
$$
\widehat{\beta}'\circ\mu\circ \widehat{\varphi}\circ\iota_{\mathbb L}\,=\,0\, .
$$
Now from the exactness of the bottom row in \eqref{e9} it follows that the image of
$\mu\circ \widehat{\varphi}\circ\iota_{\mathbb L}$ is contained in the image of the
injective map $\iota_1$ in \eqref{e9}. Therefore,
$\mu\circ \widehat{\varphi}\circ\iota_{\mathbb L}$ defines a map
\begin{equation}\label{e12}
{\mathcal R}_\varphi\, :\, {\mathbb L}\, \longrightarrow\, \text{ad}(E_H)\vert_D\, .
\end{equation}
The homomorphism ${\mathcal R}_\varphi$ in \eqref{e12} is called the \textit{residue} of
the logarithmic connection $\varphi$ \cite{De}.

\section{Manifolds with trivial logarithmic tangent bundle}

\subsection{Trivialization of the logarithmic tangent bundle}\label{se3.1}

We now assume that
\begin{enumerate}
\item $X$ is compact, and

\item the logarithmic tangent bundle $TX(-\log D)$ is holomorphically trivial.
\end{enumerate}
Such pairs $(X,\, D)$ were classified in \cite{Wi2}; this was done earlier in \cite{Wi1} 
under an extra assumption that $X$ lies in class $\mathcal C$. Below we briefly recall from 
\cite{Wi2}.

Let $X_0\,:=\,X\setminus D$ be the complement. Denote by $\mathbb G$ the connected component of the 
group of holomorphic automorphisms of $X$ that preserve $D$. This $\mathbb G$ is a complex Lie group 
and it acts transitively on $X_0$. For each point of $X_0$, the isotropy subgroup of $\mathbb G$ is 
discrete. Let $C$ denote the connected component of the center of $\mathbb G$ containing the identity 
element. It is a semi-torus, meaning $C$ is a quotient of the additive group
$({\mathbb C}^{\dim C},\, +)$ by a discrete subgroup that generates the vector space
${\mathbb C}^{\dim C}$. There is a locally holomorphically trivial fibration
$$
\pi\, :\, X\, \longrightarrow\, Y\, ,
$$
such that
\begin{itemize}
\item $Y$ is a compact parallelizable manifold, more precisely, the
quotient group ${\mathbb G}/C$ acts transitively on $Y$ with discrete cocompact isotropies,

\item the projection $\pi$ is $\mathbb G$--equivariant and it admits a holomorphic connection
preserved by the action of $\mathbb G$,

\item the typical fiber of the fiber bundle $\pi$ is an equivariant compactification of
$C$, such that all isotropy subgroups are semi-tori, and

\item $C$ is the structure group of the fiber bundle.
\end{itemize}
(See \cite[p.~196, Theorem 1]{Wi2}.)

Let $\mathfrak g$ denote the Lie algebra of the above defined group $\mathbb G$. Let
$$
{\mathcal G}\, :=\, X\times {\mathfrak g}\, \longrightarrow\, X
$$
be the trivial holomorphic vector bundle over $X$ with fiber $\mathfrak g$.

The tautological action of $\mathbb G$ on $X$ (recall that $\mathbb G$ is a subgroup of the 
group of holomorphic automorphisms of $X$) produces a homomorphism
$$
\gamma_0\, :\, {\mathfrak g} \, \longrightarrow\, H^0(X,\, TX)\, .
$$
This homomorphism $\gamma_0$ preserves the Lie algebra structures of $\mathfrak g$
and $H^0(X,\, TX)$ (its Lie algebra structure is given by Lie bracket of vector fields).
The action of $\mathbb G$ on $X$, by definition, preserves $D$, and from this it follows that
\begin{equation}\label{e19}
\gamma_0({\mathfrak g})\, \subset\, H^0(X,\, TX(-\log D))\, .
\end{equation}
Let
\begin{equation}\label{e15}
\gamma\, :\, {\mathcal G}\, :=\, X\times {\mathfrak g}\, \longrightarrow\, TX(-\log D)
\end{equation}
be the ${\mathcal O}_X$--linear homomorphism defined by 
$$
\gamma(x)(v)\,=\, \gamma_0(v)(x) \, \in\, TX(-\log D)_x \, , \ \ \forall\ x\, \in\, X\, ,
\ v\, \in\, \mathfrak g\, .
$$

\begin{lemma}\label{lem1}
The homomorphism $\gamma$ in \eqref{e15} is an isomorphism.
\end{lemma}

\begin{proof}
Since $\mathbb G$ acts transitively on $X\setminus D$, we have $\dim {\mathfrak g}\, \geq\,
\dim X\,=\, d$. Take any $x\in\, X_0\,=\,X\setminus D$. The isotropy subgroup of $\mathbb G$
for $x$ is discrete, and hence the homomorphism
$$
\gamma(x)\, :\, {\mathfrak g}\, \longrightarrow\, T_xX
$$
is injective. Now from the above inequality $\dim {\mathfrak g}\, \geq\, \dim X$ it follows immediately
that $\gamma(x)$ is an isomorphism if $x\in\, X_0$. On the other hand, both the
vector bundles ${\mathcal G}$ and $TX(-\log D)$ in \eqref{e15} are holomorphically
trivial, and $X$ is compact. From these it follows that $\gamma$ is an isomorphism. To see this, consider the
homomorphism of top exterior products
$$
\bigwedge\nolimits^d\gamma\, :\, {\mathcal O}_X \, =\, \bigwedge\nolimits^d {\mathcal G}
\, \longrightarrow\, \bigwedge\nolimits^d TX(-\log D)\,=\, {\mathcal O}_X
$$
induced by $\gamma$ in \eqref{e15}, where $d\, =\, \dim_{\mathbb C} X$. Any homomorphism ${\mathcal O}_X\, 
\longrightarrow\,{\mathcal O}_X$ is given by a globally defined holomorphic function on $X$. From the above 
observation that $\gamma$ is an isomorphism over $X_0$ it follows immediately that $\bigwedge\nolimits^d\gamma$ is 
an isomorphism over $X\setminus D$. Therefore, we conclude that $\bigwedge\nolimits^d\gamma$ corresponds to a 
nowhere zero holomorphic function on $X$. This implies that $\bigwedge\nolimits^d\gamma$ is an isomorphism over 
entire $X$. From this it follows immediately that $\gamma$ is an isomorphism over $X$.
\end{proof}

\begin{corollary}\label{cor1}
The homomorphism $$\gamma_0\, :\, {\mathfrak g}\, \, \longrightarrow\, H^0(X,\, TX(-\log D))$$
in \eqref{e19} is an isomorphism.
\end{corollary}

\begin{proof}
The global holomorphic sections of the two holomorphic vector bundles ${\mathcal G}$ and 
$TX(-\log D)$ are ${\mathfrak g}$ and $H^0(X,\, TX(-\log D))$ respectively. The homomorphism $\gamma_0$
evidently coincides with the homomorphism of global sections
$$
H^0(X,\, {\mathcal G})\, \longrightarrow\, H^0(X,\, TX(-\log D))
$$
given by $\gamma$ in \eqref{e15}. But $\gamma$ is an isomorphism by Lemma \ref{lem1}.
Consequently, $\gamma_0$ is an isomorphism.
\end{proof}

\subsection{Equivariant bundles}

Let $M$ be a connected complex Lie group and
$$
\rho\, :\, M\, \longrightarrow\, {\mathbb G}
$$
a surjective holomorphic homomorphism, where $\mathbb G$ is the group in Section
\ref{se3.1}. Let $H$ be a \textit{connected} complex Lie group.

\begin{definition}\label{def0}
A $\rho$--\textit{equivariant} principal $H$--bundle is a pair $(E_H,\, \delta)$,
where $E_H$ as in \eqref{e3} is a holomorphic principal $H$--bundle over $X$, and
$$
\delta\, :\, M\times E_H\, \longrightarrow\, E_H
$$
is a left--action of the group $M$ on $E_H$, such that the following conditions hold:
\begin{enumerate}
\item the action of $M$ is holomorphic, meaning the map $\delta$ is holomorphic,

\item the actions of $M$ and $H$ on $E_H$ commute,

\item for any $(y,\, z)\, \in\, M\times E_H$,
\begin{equation}\label{e16}
p(\delta(y,z)) \,=\, \rho(y)(p(z))\, ,
\end{equation}
where $p$ is the projection in \eqref{e3} (recall that ${\mathbb G}\, \subset\,
{\rm Aut}(X)$).
\end{enumerate}
\end{definition}

\section{A criterion for equivariance}

\begin{theorem}\label{thm1}
\mbox{}
\begin{enumerate}
\item{} Let $(E_H,\, \delta)$ be a $\rho$--equivariant holomorphic principal $H$--bundle, where 
$\rho\, :\, M\, \longrightarrow\, {\mathbb G}$ is a surjective holomorphic homomorphism of connected
complex Lie groups. Then $E_H$ admits a logarithmic connection singular over $D$.

\item{} Let $E_H$ be a holomorphic principal $H$--bundle over $X$ admitting a logarithmic
connection singular over $D$. Then there is connected complex Lie group $M$ and a surjective
holomorphic homomorphism $\rho\, :\, M\, \longrightarrow\, {\mathbb G}$, such that there is a
left--action $\delta$ of $M$ on $E_H$ with the property that $(E_H,\, \delta)$
is a $\rho$--equivariant holomorphic principal $H$--bundle.
\end{enumerate}
\end{theorem}

\begin{proof}
To prove (1), let $(E_H,\, \delta)$ be a $\rho$--equivariant holomorphic principal $H$--bundle,
where $\rho\, :\, M\, \longrightarrow\, {\mathbb G}$ is a surjective
holomorphic homomorphism of connected
complex Lie groups. The Lie algebra of $M$ will be denoted by $\mathfrak m$. Let
$$
\phi\, :\, {\mathfrak m}\, \longrightarrow\, H^0(E_H,\, TE_H)
$$
be the homomorphism given by the action $\delta$ of $M$ on $E_H$. From the given condition, that
the actions of $M$ and $H$ on $E_H$ commute, it follows immediately that
$$
\phi({\mathfrak m})\, \subset\, H^0(E_H,\, TE_H)^H\, ,
$$
where $H^0(E_H,\, TE_H)^H\, \subset\, H^0(E_H,\, TE_H)$ is the space of $H$--invariant holomorphic
vector fields on $E_H$ for the natural action of $H$ on $E_H$. From the definition ${\rm At}(E_H)
\,:=\, (TE_H)/H$ it evidently follows that
$$
H^0(E_H,\, TE_H)^H\,=\, H^0(X, \, {\rm At}(E_H))\, ,
$$
so we have $\phi({\mathfrak m})\, \subset\, H^0(X, \, {\rm At}(E_H))$. Since
the action of $\mathbb G$ on $X$, by definition, preserves $D$, from \eqref{e16}
it follows that the action of $M$ on $E_H$ preserves the inverse image
$\widehat{D}\,=\, p^{-1}(D)$ in \eqref{e5}. Hence we have
$$
\phi\, :\, {\mathfrak m}\, \longrightarrow\, H^0(E_H,\, TE_H(-\log \widehat{D}))^H
\, \subset\, H^0(E_H,\, TE_H)^H\, .
$$
From \eqref{e6} it now follows that
\begin{equation}\label{e17}
\phi({\mathfrak m}) \, \subset\, H^0(X, \, {\rm At}(E_H)(-\log D))
\, \subset\, H^0(X, \, {\rm At}(E_H))\, .
\end{equation}

Let
\begin{equation}\label{rp}
\rho'\, :\, {\mathfrak m}\, \longrightarrow\, {\mathfrak g}
\end{equation}
be the homomorphism
of Lie algebras for the homomorphism $\rho$ of Lie groups. From \eqref{e16} it follows
that for all $w\, \in\, \mathfrak m$, we have
$$
\beta \circ\phi(w)\,=\, \gamma_0\circ\rho'(w)\, \in\, H^0(X,\, TX(-\log D))\, ,
$$
where $\beta$ and $\gamma_0$ are the homomorphisms in \eqref{e8} and
Corollary \ref{cor1} respectively.
The homomorphism $\rho'$ is surjective because $\rho$ is so. Combining this with Corollary \ref{cor1}
it follows that the homomorphism
\begin{equation}\label{e29}
\gamma_0\circ\rho'\, :\, {\mathfrak m}\, \longrightarrow\, H^0(X,\, TX(-\log D))
\end{equation}
is surjective. Fix a $\mathbb C$--linear map
\begin{equation}\label{xi}
\xi\, :\, H^0(X,\, TX(-\log D))\, \longrightarrow\, {\mathfrak m}
\end{equation}
such that
\begin{equation}\label{e20}
(\gamma_0\circ\rho')\circ\xi\,=\, \text{Id}_{H^0(X, TX(-\log D))}\, .
\end{equation}
Now consider the composition
\begin{equation}\label{e22}
\phi\circ\xi\, :\, H^0(X,\, TX(-\log D))\, \longrightarrow\, H^0(X, \, {\rm At}(E_H)(-\log D))\, ,
\end{equation}
where $\phi$ and $\xi$ are the maps in \eqref{e17} and \eqref{xi} respectively. From \eqref{e20}
it follows that
\begin{equation}\label{e21}
\beta_* \circ (\phi\circ\xi)\,=\, \text{Id}_{H^0(X, TX(-\log D))}\, ,
\end{equation}
where
\begin{equation}\label{bts}
\beta_*\, :\, H^0(X, \, {\rm At}(E_H)(-\log D))\,\longrightarrow\, H^0(X,\, TX(-\log D))
\end{equation}
is the homomorphism of global sections induced by the map $\beta$ of vector bundles in \eqref{e8}.

Since the vector bundle $TX(-\log D)$ is holomorphically trivial, the map $\phi\circ\xi$
in \eqref{e22} defines a homomorphism
\begin{equation}\label{wxi}
\widetilde{\xi}\, :\, TX(-\log D)\, \longrightarrow\, {\rm At}(E_H)(-\log D)\, .
\end{equation}
To construct this map $\widetilde{\xi}$,
for any $x\,\in\, X$ and $v\,\in\, TX(-\log D)_x$, let $$\widetilde{v}\,
\in\, H^0(X,\, TX(-\log D))$$ be the unique holomorphic section such that $\widetilde{v}(x)\,=\, v$. The map
$\widetilde{\xi}$ sends $v$ to $\phi\circ\xi(\widetilde{v})(x)\,\in\, {\rm At}(E_H)
(-\log D)_x$. From \eqref{e21} it follows immediately that
$$
\beta\circ \widetilde{\xi}\,=\, \text{Id}_{TX(-\log D)}\, .
$$
In other words, $\widetilde{\xi}$ is a logarithmic connection on $E_H$ singular over $D$.
This proves (1) in the theorem.

We shall now prove (2) in the theorem.

Let $\mathbb A$ denote the space of all pairs of the form $(\tau,\, f)$, where
$\tau\, :\, X\, \longrightarrow\, X$ is a biholomorphism and
$$
f\, :\, E_H\, \longrightarrow\, E_H
$$
is a biholomorphism such that
\begin{itemize}
\item $\tau\circ p\,\,=\, p\circ f$, where $p$ is the projection in \eqref{e3}, and

\item $f(q'(z,h))\,=\, q'(f(z),h)$, for all $(z,\, h)\, \in\, E_H\times H$, where $q'$ is
the action in \eqref{e4} (in other words, $f$ is $H$--equivariant).
\end{itemize}
We note that $\mathbb A$ is a group, with group operation map and inverse map respectively given by
$$
(\tau_1,\, f_1)\cdot (\tau_2,\, f_2)\,=\, (\tau_1\circ\tau_2,\, f_1\circ f_2)\ \
\text{ and }\ \ (\tau,\, f)^{-1}\,=\, (\tau^{-1},\, f^{-1})\, .
$$
This $\mathbb A$ is a complex Lie group with Lie algebra $H^0(X,\, \text{At}(E_H))$
(the Lie algebra structure on $H^0(X,\, \text{At}(E_H))$ is given by the
Lie bracket of vector fields).

Let
$$
{\mathbb A}_D\, \subset\, \mathbb A
$$
denote the subgroup consisting of all $(\tau,\, f)$ of the above type such that
$\tau(D)\,=\, D$. It is a complex Lie subgroup with Lie algebra
$H^0(X,\, \text{At}(E_H)(-\log D))$. Let
$$
{\mathbb A}^0_D\, \subset\, {\mathbb A}_D
$$
be the connected component containing the identity element. Define a homomorphism
\begin{equation}\label{e23}
\theta\, :\, {\mathbb A}^0_D\, \longrightarrow\, {\mathbb G}\, ,\ \ (\tau,\, f)\, \longmapsto
\, \tau\, ,
\end{equation}
where $\mathbb G$ is the group defined in Section \ref{se3.1}.

Now assume that $E_H$ admits a logarithmic connection singular over $D$. We shall show that
the homomorphism $\theta$ in \eqref{e23} is surjective.

To prove that $\theta$ is surjective,
fix a logarithmic connection $$\varphi\, :\, TX(-\log D) \, \longrightarrow\, \text{At}(E_H)
(-\log D)$$ on $E_H$ singular over $D$. Let
\begin{equation}\label{e24}
\varphi_*\, :\, H^0(X,\, TX(-\log D)) \, \longrightarrow\, H^0(X,\, \text{At}(E_H)(-\log D))
\end{equation}
be the homomorphism of global sections induced by $\varphi$.

Let
\begin{equation}\label{e25}
\theta'\, :\, \text{Lie}({\mathbb A}^0_D)\,=\, H^0(X,\, \text{At}(E_H)(-\log D))
\, \longrightarrow\, {\mathfrak g}
\end{equation}
be the homomorphism of Lie algebras associated to the homomorphism $\theta$ in
\eqref{e23}. To prove that $\theta$ is surjective it suffices to show that $\theta'$ is
surjective, because $\mathbb G$ is connected.

Now consider the composition
$$
\theta'\circ\varphi_*\, :\, H^0(X,\, TX(-\log D)) \, \longrightarrow\, \mathfrak g\, ,
$$
where $\varphi_*$ and $\theta'$ are constructed in \eqref{e24} and \eqref{e25}
respectively. From the constructions of $\theta'\circ\varphi_*$ and the map $\gamma_0$
(see Corollary \ref{cor1}) it follow immediately that
$$
(\theta'\circ\varphi_*)\circ \gamma_0\,=\, \text{Id}_{\mathfrak g}\, .
$$
On the other hand, from Corollary \ref{cor1} we know that the homomorphism
$\gamma_0$ is an isomorphism. Hence $\theta'\circ\varphi_*$ is also an isomorphism. This
implies that $\theta'$ is surjective. As noted before, the surjectivity of the
homomorphism $\theta$ in \eqref{e23} follows from the surjectivity of $\theta'$.

The group ${\mathbb A}^0_D$ in \eqref{e23} has a tautological holomorphic action on $E_H$ 
(recall that it is a subgroup of the automorphism group of the complex manifold $E_H$); let
\begin{equation}\label{ta}
{\mathbb T}_D\, :\, {\mathbb A}^0_D\times E_H\, \longrightarrow\, E_H
\end{equation}
be this tautological action of ${\mathbb A}^0_D$ on $E_H$. The
pair $(E_H,\, {\mathbb T}_D)$ is evidently a $\theta$--equivariant holomorphic principal
$H$--bundle, where $\theta$ is the homomorphism in \eqref{e23}. This completes the proof of (2).
\end{proof}

As in Theorem \ref{thm1}(1), let $(E_H,\, \delta)$ be a $\rho$--equivariant holomorphic 
principal $H$--bundle, where $\rho\, :\, M\, \longrightarrow\, {\mathbb G}$ is a surjective 
holomorphic homomorphism of connected complex Lie groups. Fix a homomorphism $\xi$ as in
\eqref{xi} satisfying the condition in \eqref{e20}. Construct the homomorphism
$\widetilde{\xi}$ in \eqref{wxi} using $\xi$. It was shown in the proof of
Theorem \ref{thm1}(1) that $\widetilde{\xi}$ is a logarithmic connection on $E_H$ singular
on $D$. Let
\begin{equation}\label{kxi}
{\mathbb K}(\widetilde{\xi})\, \in\, H^0(X,\, \Omega^2_X(\log D)\otimes {\rm ad}(E_H))
\end{equation}
be the curvature of the logarithmic connection on $E_H$ defined by $\widetilde{\xi}$;
curvature was defined in \eqref{e13}.

We shall compute ${\mathbb K}(\widetilde{\xi})$ in \eqref{kxi}. For that, consider the linear
map
\begin{equation}\label{k0}
K_0\, :\, \bigwedge\nolimits^2 H^0(X,\, TX(-\log D)) \, \longrightarrow\, {\mathfrak m}\, ,
\ \ v\wedge w \, \longmapsto\, [\xi(v),\, \xi(w)] -\xi([v,\, w])\, ;
\end{equation}
so $K_0$ measures how the $\mathbb C$--linear homomorphism $\xi$ fails to be a homomorphism
of Lie algebras. We have the composition homomorphism
$$
(\beta_*\circ\phi)\circ K_0\, :\, \bigwedge\nolimits^2 H^0(X,\, TX(-\log D)) \, \longrightarrow\,
H^0(X,\, TX(-\log D))\, ,
$$
where $\beta_*$ is the homomorphism in \eqref{bts} of global sections given
by $\beta$ in \eqref{e8}, and $\phi$ is the homomorphism in \eqref{e17}. From \eqref{e20}, and the
fact that both $\phi$ and $\beta$ are Lie algebra structure preserving, it
follows immediately that
$$
(\beta_*\circ\phi)\circ K_0\,=\, 0\, .
$$
Hence from the short exact sequence in \eqref{e8} it now follows that
$$
\phi\circ K_0(\bigwedge\nolimits^2 H^0(X,\, TX(-\log D)))\, \subset\, H^0(X,\, \text{ad}(E_H))
$$
$$
\subset\, H^0(X, \, {\rm At}(E_H)(-\log D))\, .
$$
Recall that the holomorphic vector bundle $TX(-\log D)$ is holomorphically trivial. So
the holomorphic vector bundle $\bigwedge\nolimits^2 (TX(-\log D))$ is also
holomorphically trivial, and, moreover, we have
$$
\bigwedge\nolimits^2 H^0(X,\, TX(-\log D))\,=\, H^0(X,\, \bigwedge\nolimits^2 (TX(-\log D)))\, .
$$
So the above homomorphism $\phi\circ K_0$ can be considered as a homomorphism
\begin{equation}\label{e26}
\phi\circ K_0\, :\, H^0(X,\, \bigwedge\nolimits^2 (TX(-\log D)))\, \longrightarrow\,
H^0(X,\, \text{ad}(E_H))\, .
\end{equation}
Since $\bigwedge\nolimits^2 (TX(-\log D))$ is holomorphically trivial, the homomorphism
$\phi\circ K_0$ in \eqref{e26} produces a homomorphism of coherent analytic sheaves
$$
(\phi\circ K_0)'\, :\, \bigwedge\nolimits^2 (TX(-\log D))\, \longrightarrow\,
\text{ad}(E_H)
$$
as follows: for any $x\, \in\, X$ and $w \, \in\, \bigwedge\nolimits^2 (TX(-\log D))_x$, let
$\widetilde{w}$ be the unique element of $H^0(X,\, \bigwedge\nolimits^2 (TX(-\log D)))$ satisfying
the condition that $\widetilde{w}(x)\,=\, w$. Now set
$$
(\phi\circ K_0)'(x)(w)\,=\, (\phi\circ K_0)(\widetilde{w})(x)\, \in\, \text{ad}(E_H)_x\, .
$$
The above homomorphism $(\phi\circ K_0)'$ defines an element
\begin{equation}\label{e27}
\sigma\, \in\, H^0(X,\, \text{Hom}(\bigwedge\nolimits^2 (TX(-\log D)),\, \text{ad}(E_H)))\,=\,
H^0(X,\, \Omega^2_X(\log D)\otimes {\rm ad}(E_H))\, .
\end{equation}

{}From the definition of ${\mathbb K}(\widetilde{\xi})$ in \eqref{kxi} (see \eqref{e13}) and
the construction of $\sigma$ in \eqref{e27} it is straightforward to check that
\begin{equation}\label{e28}
\sigma\, =\, {\mathbb K}(\widetilde{\xi})\, .
\end{equation}

The residue of the logarithmic connection $\widetilde\xi$ (constructed
in \eqref{e12}) can also be described in terms of $\xi$.

\begin{lemma}\label{lem2}
Let $(E_H,\, \delta)$ be a $\rho$--equivariant holomorphic
principal $H$--bundle, where $\rho\, :\, M\, \longrightarrow\, {\mathbb G}$ is a surjective
holomorphic homomorphism of connected complex Lie groups, such that the homomorphism
$\gamma_0\circ\rho'$ in \eqref{e29} is an isomorphism. Then $E_H$ admits a flat holomorphic
connection singular over $D$ which is preserved by the action $\delta$ of $M$ on $E_H$.
\end{lemma}

\begin{proof}
Set $\xi$ in \eqref{xi} to be the inverse of the isomorphism $\gamma_0\circ\rho'$. Note that 
as $\gamma_0\circ\rho'$ is an isomorphism, the condition in \eqref{e20} forces $\xi$ to be 
the inverse of $\gamma_0\circ\rho'$. Therefore, $\xi$ is an isomorphism of Lie algebras. Hence 
$K_0$ in \eqref{k0} vanishes. Now from \eqref{e28} it follows that connection $\widetilde\xi$ in 
\eqref{wxi} is flat. Since the image of $\xi$ in $\mathfrak m$ is preserved by the adjoint 
action of $M$ on $\mathfrak m$ (in fact the image of $\xi$ is entire $\mathfrak m$), it is 
straightforward to deduce that the action of $M$ on $E_H$ preserves the connection 
$\widetilde\xi$.
\end{proof}

\begin{definition}\label{def1}
A holomorphic principal $H$--bundle $E_H$ over $X$ is called \textit{homogeneous} if there is a 
connected complex Lie group $M$ and a surjective holomorphic homomorphism $\rho\, :\, M\, 
\longrightarrow\, {\mathbb G}$, such that there is a holomorphic action $\delta$ of $M$ on $E_H$ 
satisfying the condition that $(E_H,\, \delta)$ is a $\rho$--equivariant holomorphic principal 
$H$--bundle.
\end{definition}

\begin{proposition}\label{prop2}
A holomorphic principal $H$--bundle $E_H$ over $X$ is homogeneous if and only if
the holomorphic principal $H$--bundle $g^*E_H$ is holomorphically isomorphic to $E_H$
for every $g\, \in\, \mathbb G$.
\end{proposition}

\begin{proof}
First assume that there is a
connected complex Lie group $M$ and a surjective holomorphic homomorphism $\rho\, :\, M\,
\longrightarrow\, {\mathbb G}$, such that there is a holomorphic action $\delta$ of $M$ on $E_H$
satisfying the condition that $(E_H,\, \delta)$ is a $\rho$--equivariant holomorphic principal
$H$--bundle. For any $g\, \in\, \mathbb G$, take any $\widetilde{g}\, \in\, \rho^{-1}(g)$. Then the
action of $\widetilde{g}$ on $E_H$ (given by $\delta$) is a holomorphic isomorphism of the
holomorphic principal $H$--bundle $g^*E_H$ with $E_H$. 

To prove the converse, assume that $g^*E_H$ is holomorphically isomorphic to $E_H$ for every $g\, \in\, \mathbb 
G$. Consider the homomorphism $\theta$ in \eqref{e23}. Since $g^*E_H$ is holomorphically isomorphic to $E_H$ for 
every $g\, \in\, \mathbb G$, we conclude that the homomorphism $\theta$ is surjective. Now in Definition 
\ref{def1}, set $M\,=\, {\mathbb A}^0_D$, $\rho\,=\, \theta$ and $\delta$ to be the tautological action ${\mathbb 
T}_D$ of ${\mathbb A}^0_D$ on $E_H$ (see \eqref{ta}). Since $\theta$ is surjective, it follows that $E_H$ is 
homogeneous.
\end{proof}

The following is an immediate consequence of Theorem \ref{thm1} and Proposition \ref{prop2}.

\begin{corollary}\label{cor2}
A holomorphic principal $H$--bundle $E_H$ over $X$ admits a logarithmic connection singular over $D$
if and only if for every $g\, \in\, \mathbb G$ the holomorphic principal $H$--bundle $g^*E_H$ is
holomorphically isomorphic to $E_H$.
\end{corollary}

\section{Infinitesimal deformations of principal bundles}

The space of all infinitesimal deformations of a holomorphic principal
$H$--bundle $E_H\,\stackrel{p}{\longrightarrow}\, X$ are parametrized by $H^1(X,\,
\text{ad}(E_H))$. This in particular means the following.
Let $T$ be a complex manifold with a base point $t_0\, \in\, T$. Let
${\mathcal E}_H\, \longrightarrow\, T\times X$ be a holomorphic principal $H$--bundle.
For any $t\, \in\, T$, the holomorphic principal $H$--bundle ${\mathcal E}_H\vert_{\{t\}\times X}$
on $X$ will be denoted by ${\mathcal E}^t_H$. Assume that we are given a holomorphic isomorphism
of $E_H$ with the principal $H$--bundle ${\mathcal E}^{t_0}_H$; here $X$ is identified with
$\{t_0\}\times X$ using the map $x\, \longmapsto\, (t_0,\, x)$. So $\{{\mathcal E}^t_H\}_{t\in T}$
is a holomorphic family of holomorphic principal $H$--bundles over $X$ such that the
holomorphic principal $H$--bundle over $X$ for $t_0\, \in\, T$ is the given holomorphic principal
$H$--bundle $E_H$. Then there is a $\mathbb C$--linear homomorphism
\begin{equation}\label{d1}
T_{t_0} T\, \longrightarrow\, H^1(X,\, \text{ad}(E_H))
\end{equation}
which is compatible with the pullback operation of families of holomorphic principal $H$--bundles 
on $X$ (see \cite{Do} for more details). The homomorphism in \eqref{d1} is called the
\textit{infinitesimal deformation map} at $t_0$ for the family ${\mathcal E}_H$ of holomorphic
principal $H$--bundles over $X$.

Consider the Lie group $\mathbb G$ in Section \ref{se3.1}. Let
\begin{equation}\label{vp}
\varpi\, :\, {\mathbb G}\times X\, \longrightarrow\, X\, , \ \ (g,\, x)\, \longmapsto\, g(x)
\end{equation}
be the evaluation map (recall that $\mathbb G$ is a subgroup of the automorphism group of $X$).
Note that this map $\varpi$ is holomorphic. As in \eqref{e3}, consider a holomorphic principal
$H$--bundle $E_H\,\stackrel{p}{\longrightarrow}\, X$. Let
\begin{equation}\label{d3}
{\mathcal E}_H\, :=\, \varpi^*E_H\, \longrightarrow\, {\mathbb G}\times X
\end{equation}
be the pulled back holomorphic principal $H$--bundle over ${\mathbb G}\times X$. Note
that the restriction ${\mathcal E}^e_H \,=\, (\varpi^*E_H)\vert_{\{e\}\times X}$,
where $e$ is the identity element of $\mathbb G$, is identified with the
holomorphic principal $H$--bundle $E_H$. Therefore, as in \eqref{d1}, we have
the infinitesimal deformation map
\begin{equation}\label{d2}
f\, :\, T_e{\mathbb G}\,=\, {\mathfrak g} \, \longrightarrow\, H^1(X,\, \text{ad}(E_H))\, .
\end{equation}

\begin{proposition}\label{prop1}
A holomorphic principal $H$--bundle $E_H$ over $X$ is homogeneous (see Definition \ref{def1})
if and only if the homomorphism $f$ in \eqref{d2} is the zero map.
\end{proposition}

\begin{proof}
First assume that $E_H$ is homogeneous. So there is a
connected complex Lie group $M$ and a surjective holomorphic homomorphism $\rho\, :\, M\,
\longrightarrow\, {\mathbb G}$, such that there is a holomorphic action
\begin{equation}\label{d8}
\delta\, :\, M\times E_H\, \longrightarrow\, E_H
\end{equation}
as in Definition \ref{def0} satisfying the condition that $(E_H,\, \delta)$ is a
$\rho$--equivariant holomorphic principal $H$--bundle. We have the principal $H$--bundle
\begin{equation}\label{d7}
{\mathcal F}_H\,:=\, (\rho\times\text{Id}_X)^* {\mathcal E}_H\, \longrightarrow\,
M\times E_H\, ,
\end{equation}
where ${\mathcal E}_H$ is the holomorphic principal $H$--bundle in \eqref{d3}. Let
$e'\, \in\, M$ be the identity element. For the family of principal $H$--bundles
${\mathcal F}_H$ parametrized by $M$,
as in \eqref{d1}, we have the infinitesimal deformation map
\begin{equation}\label{d5}
f_1\, :\, T_{e'} M\,=\, {\mathfrak m} \, \longrightarrow\, H^1(X,\, \text{ad}(E_H))\, ,
\end{equation}
where $\mathfrak m$ as before is the Lie algebra of $M$. Since the homomorphism in
\eqref{d1} is compatible with the pullback operation of families of holomorphic principal
$H$--bundles over $X$, we have
\begin{equation}\label{d6}
f_1\,=\, f\circ\rho'\, ,
\end{equation}
where $f$, $f_1$ and $\rho'$ are the homomorphisms in \eqref{d2}, \eqref{d5} and
\eqref{rp} respectively.

It can be shown that the action $\delta$ in \eqref{d8} identifies the family of holomorphic
principal $H$--bundles ${\mathcal F}_H$ in \eqref{d7} with the trivial family
\begin{equation}\label{d9}
p^*_2 E_H\, \longrightarrow\, M\times X\, ,
\end{equation}
where $p_2\, :\, M\times X\, \longrightarrow\, X$ is the projection
$(z,\, x)\, \longmapsto\, x$. To prove this, consider the map
$$
\widetilde{\delta}\, :\, M\times E_H\, \longrightarrow\, M\times E_H\, ,\ \ (z,\, y)\,
\longmapsto\, (z,\, \delta(z,y))\, .
$$
From the equality in \eqref{e16} it follows that this map $\widetilde{\delta}$ is a holomorphic
isomorphism of the holomorphic principal $H$--bundle ${\mathcal F}_H$ in \eqref{d7} with the
holomorphic principal $H$--bundle $p^*_2 E_H$ in \eqref{d9}. Since $p^*_2 E_H$ is a constant
family, the infinitesimal deformation map 
$$
T_{e'} M\,=\, {\mathfrak m} \, \longrightarrow\, H^1(X,\, \text{ad}(E_H))
$$
for $p^*_2 E_H$ is the zero homomorphism. Consequently,
the infinitesimal deformation map $f_1$ in \eqref{d5} is the zero homomorphism. Since
$\rho'$ in \eqref{rp} is surjective (as $\rho$ is surjective) from \eqref{d6} it now follows
that $f\,=\, 0$.

To prove the converse, assume that $f\,=\, 0$. Consider the short exact sequence in
\eqref{e8}. Let
\begin{equation}\label{d10}
H^0(X,\, {\rm At}(E_H)(-\log D))\, \stackrel{\beta_*}{\longrightarrow}\,
H^0(X,\,TX(-\log D)) \, \stackrel{\lambda}{\longrightarrow}\,H^1(X,\, \text{ad}(E_H))
\end{equation}
be the long exact sequence of cohomologies associated to it, where $\beta_*$ is the
homomorphism in \eqref{bts} and $\lambda$ is the connecting homomorphism. Consider the isomorphism
$$
\gamma_0\, :\, {\mathfrak g}\, \, \longrightarrow\, H^0(X,\, TX(-\log D))
$$
in Corollary \ref{cor1}. We have
$$
\lambda\circ\gamma_0\,=\, f\, ,
$$
where $\lambda$ is the homomorphism in \eqref{d10}. Since $f\,=\, 0$, and $\gamma_0$ is
an isomorphism, we conclude that $\lambda\, =\, 0$.

The homomorphism $\beta_*$ in \eqref{d10} is surjective, because $\lambda\,=\, 0$.
Fix a homomorphism
$$
b\, :\, H^0(X,\,TX(-\log D))\, \longrightarrow\,H^0(X,\, {\rm At}(E_H)(-\log D))
$$
such that
\begin{equation}\label{d11}
\beta_*\circ b\,=\,\text{Id}_{H^0(X,TX(-\log D))}\, .
\end{equation}
As the holomorphic vector bundle $TX(-\log D)$ is holomorphically trivial, the homomorphism
$b$ in \eqref{d11} produces a homomorphism
$$
b'\, :\, TX(-\log D)\, \longrightarrow\, {\rm At}(E_H)(-\log D)
$$
that sends any $v\, \in\, TX(-\log D)_x$ to $$b(\widetilde{v})(x)\,\in\, {\rm At}(E_H)
(-\log D)_x\, ,$$ where $\widetilde{v}\, \in\, H^0(X,\,TX(-\log D))$ is the unique
holomorphic section such that $\widetilde{v}(x)\,=\, v$. From \eqref{d11} it follows that
$$
\beta\circ b'\,=\, \text{Id}_{TX(-\log D)}\, ,
$$
where $\beta$ is the projection in \eqref{e8}. Hence $b'$ defines a logarithmic connection
on $E_H$ singular over $D$. Now from Theorem \ref{thm1}(2) we conclude that $E_H$
is homogeneous.
\end{proof}

\section*{Acknowledgements}

The second-named author thanks Universit\'e de Cergy-Pontoise for hospitality.
He is partially supported by a J. C. Bose Fellowship. The third-named
author is grateful to ASSMS, GC University Lahore for the support of this research under
the postdoctoral fellowship.

\end{document}